\documentclass[11pt]{amsart}
\usepackage{amsopn}
\usepackage{amssymb, amscd}
\usepackage{multirow}
\usepackage{verbatim}
\usepackage{easybmat}
\usepackage{graphicx}
\usepackage{arydshln}

\newcommand{\nc}{\newcommand}

\nc{\fg}{\mathfrak{f} } \nc{\vg}{\mathfrak{v} } \nc{\wg}{\mathfrak{w} }
\nc{\zg}{\mathfrak{z} } \nc{\ngo}{\mathfrak{n} } \nc{\kg}{\mathfrak{k} }
\nc{\mg}{\mathfrak{m} } \nc{\lgo}{\mathfrak{l} } \nc{\ggo}{\mathfrak{g} }
\nc{\ggob}{\overline{\mathfrak{g}} } \nc{\sog}{\mathfrak{so} }
\nc{\sug}{\mathfrak{su} } \nc{\spg}{\mathfrak{sp} } \nc{\slg}{\mathfrak{sl} }
\nc{\glg}{\mathfrak{gl} } \nc{\cg}{\mathfrak{c} } \nc{\rg}{\mathfrak{r} }
\nc{\hg}{\mathfrak{h} } \nc{\tg}{\mathfrak{t} } \nc{\ug}{\mathfrak{u} }
\nc{\dg}{\mathfrak{d} } \nc{\ag}{\mathfrak{a} } \nc{\pg}{\mathfrak{p} }
\nc{\sg}{\mathfrak{s} } \nc{\affg}{\mathfrak{aff} }\nc{\bg}{\mathfrak{b} }

\nc{\pca}{\mathcal{P}} \nc{\nca}{\mathcal{N}} \nc{\lca}{\mathcal{L}}
\nc{\oca}{\mathcal{O}} \nc{\mca}{\mathcal{M}} \nc{\tca}{\mathcal{T}}
\nc{\aca}{\mathcal{A}} \nc{\cca}{\mathcal{C}} \nc{\gca}{\mathcal{G}}
\nc{\sca}{\mathcal{S}} \nc{\hca}{\mathcal{H}} \nc{\bca}{\mathcal{B}}
\nc{\dca}{\mathcal{D}} \nc{\val}{\operatorname{val}}

\nc{\vp}{\varphi} \nc{\ddt}{\frac{d}{dt}} \nc{\dds}{\frac{d}{ds}}
\nc{\dpar}{\frac{\partial}{\partial t}} \nc{\im}{\mathtt{i}}

\nc{\SO}{\mathrm{SO}} \nc{\Sp}{\mathrm{Sp}} \nc{\Sl}{\mathrm{SL}}
\nc{\SU}{\mathrm{SU}} \nc{\Or}{\mathrm{O}} \nc{\U}{\mathrm{U}} \nc{\Gl}{\mathrm{GL}}
\nc{\Se}{\mathrm{S}} \nc{\Cl}{\mathrm{Cl}} \nc{\Spein}{\mathrm{Spin}}
\nc{\Pin}{\mathrm{Pin}} \nc{\G}{\mathrm{GL}_n(\RR)} \nc{\g}{\mathfrak{gl}_n(\RR)}
\nc{\Span}{\mathrm{Span}}

\nc{\RR}{{\Bbb R}} \nc{\HH}{{\Bbb H}} \nc{\CC}{{\Bbb C}} \nc{\ZZ}{{\Bbb Z}}
\nc{\FF}{{\Bbb F}} \nc{\NN}{{\Bbb N}} \nc{\QQ}{{\Bbb Q}} \nc{\PP}{{\Bbb P}}

\nc{\vs}{\vspace{.2cm}} \nc{\vsp}{\vspace{1cm}} \nc{\ip}{\langle\cdot,\cdot\rangle}
\nc{\ipp}{(\cdot,\cdot)} \nc{\la}{\langle} \nc{\ra}{\rangle} \nc{\unm}{\tfrac{1}{2}}
\nc{\unc}{\tfrac{1}{4}} \nc{\und}{\tfrac{1}{16}} \nc{\no}{\vs\noindent}
\nc{\lam}{\Lambda^2(\RR^n)^*\otimes\RR^n} \nc{\tangz}{{\rm T}^{\rm Zar}}
\nc{\nor}{{\sf n}}  \nc{\mum}{/\!\!/} \nc{\kir}{/\!\!/\!\!/}
\nc{\Ri}{\tfrac{4\Ric_{\mu}}{||\mu||^2}} \nc{\ds}{\displaystyle}
\nc{\ben}{\begin{enumerate}} \nc{\een}{\end{enumerate}} \nc{\f}{\frac}
\nc{\lb}{[\cdot,\cdot]} \nc{\isn}{\tfrac{1}{||v||^2}}
\nc{\gkp}{(\ggo=\kg\oplus\pg,\ip)} \nc{\ukh}{(\ug=\kg\oplus\hg,\ip)}
\nc{\tgkp}{(\tilde{\ggo}=\kg\oplus\pg,\ip)}
\nc{\wt}{\widetilde} \nc{\mm}{M}
\renewcommand{\i}{\mathbf{i}}
\nc{\Hm}{{H}}

\nc{\ad}{\operatorname{ad}}
\nc{\Ad}{\operatorname{Ad}}
 \nc{\End}{\operatorname{End}}
\nc{\Id}{\operatorname{Id}}
\nc{\Der}{\operatorname{Der}} \nc{\Ker}{\operatorname{Ker}}
\nc{\Iso}{\operatorname{I}} \nc{\Diff}{\operatorname{Diff}}
\nc{\tr}{\operatorname{tr}}
\nc{\sen}{\operatorname{sen}}
 \nc{\Cric}{\operatorname{P}} \nc{\Ricci}{\operatorname{Ric}}
\nc{\Sym}{\operatorname{Sym}}
 \nc{\ricci}{\operatorname{Rc}}
 \nc{\Rin}{\operatorname{M}}
\nc{\Spam}{\operatorname{Spam}} \nc{\Diag}{\operatorname{Diag}}
\nc{\Spec}{\operatorname{Spec}}

\theoremstyle{plain}
\newtheorem{theorem}{Theorem}[section]
\newtheorem{proposition}[theorem]{Proposition}

\newtheorem{lemma}[theorem]{Lemma}

\theoremstyle{definition}
\newtheorem{definition}[theorem]{Definition}

\theoremstyle{remark}
\newtheorem{remark}[theorem]{Remark}

\newtheorem{example}[theorem]{Example}

\title[]{Negative Ricci Curvature on some non-solvable Lie groups II}

\author{Cynthia Will}

\address{Universidad Nacional de C\'ordoba, FaMAF and CIEM, 5000 C\'ordoba, Argentina}
\email{cwill@famaf.unc.edu.ar}

\thanks{This research was partially supported by grants from CONICET, FONCYT and SeCyT (Universidad Nacional de C\'ordoba)}

\begin{document}

\maketitle

\begin{abstract} We construct many examples of Lie groups admitting a left-invariant metric of negative Ricci curvature. We study Lie algebras which are semidirect products $\lgo = (\ag \oplus \ug) \ltimes \ngo$ and we obtain examples where $\ug$ is any semisimple compact real Lie algebra, $\ag$ is one-dimensional and $\ngo$ is a representation of $\ug$ which satisfies some conditions. In particular, when $\ug = \sug(m)$, $\sog(m)$ or $\spg(m)$ and $\ngo$ is a representation of $\ug$ in some space of homogeneous polynomials, we show that these conditions are indeed satisfied. In the case $\ug = \sug(2)$ we get a more general construction where $\ngo$ can be any nilpotent Lie algebra where $\sug(2)$ acts by derivations. We also prove a general result in the case when $\ug$ is a semisimple Lie algebra of non-compact type.
\end{abstract}

\section{Introduction}\label{intro}
 In this paper, we are interested in homogeneous negative Ricci curvature, as a continuation of the work started in \cite{u2}.
It is proved there that if $V$ is a non-trivial real representation of $\sug(2)$ extended to $\ug(2)$ by letting the center act as multiples of the identity, then the Lie algebra $\ug(2)\ltimes V$ admits an inner product with negative Ricci curvature.
Before that, the only Lie groups in the literature that were known to admit a left-invariant metric with negative Ricci curvature were either semisimple (see \cite{DL}, \cite{DLM}) or solvable (see \cite{D}, \cite{NN}, \cite{N}, \cite{LD}). We refer to \cite{NN}, \cite{u2} or \cite{LD} for a more detailed summary of the known results on negative Ricci curvature in the homogeneous case.

In this work, we extend the results in \cite{u2} in many ways, finding families of examples of Lie groups admitting a left-invariant metric with negative Ricci curvature. We construct Lie algebras as semidirect products $\lgo = (\ag \oplus \ug) \ltimes \ngo$, where $\ug$ is a semisimple Lie algebra, $\ag$ is abelian and $\ngo$ is a nilpotent Lie algebra.

First, we consider any compact semisimple Lie algebra $\ug$, a one dimensional $\ag$ and a finite dimensional real representation of $\ug$, $(V,\pi)$. We obtain a technical result showing that $\lgo=(\ag \oplus \ug) \ltimes V$ admits an inner product with negative Ricci operator provided $(V,\pi)$ admits a decomposition $V=V_1 \oplus V_2$ such that the action of the root vectors of the complexification of $\ug$ satisfy certain conditions. See Theorem \ref{theTheo} for a detailed statement.

To show that this general procedure actually gives new examples of negatively Ricci curved Lie groups, we study in particular the cases when $\ug$ is the compact real form of a classical simple Lie algebra and $(V,\pi)$ is a canonical representation of $\ug$ on a space of homogeneous polynomials.  We show that, in all these cases, the hypotheses of Theorem \ref{theTheo} are satisfied, obtaining the following results.

\begin{theorem}
  Let $\ug = \sug(m)$, $\sog(m)$  or $\spg(m)$ and let $V= \mathcal{P}_{n}(\CC^r)$ be the standard real representation of $\ug$ on the space of complex homogeneous polynomials of degree $n$ in $r$ variables, where $r=m$ for $\sug(m)$ and $\sog(m)$ and $r=2m$ for $\spg(m)$. If $\lgo = (\RR Z \oplus \ug) \ltimes V$ is the Lie algebra such that $[Z, \ug]=0$ and $Z$ acts as the identity on $V$, then  $\lgo$ admits an inner product with negative Ricci curvature for any $n,m \ge 2$.
\end{theorem}

We also consider any algebra with Levi factor $\sug(2)$.

\begin{theorem}
  Let $\lgo= (\RR Z \oplus \sug(2)) \ltimes \ngo$ be a Lie algebra where $\ngo$ is any nilpotent Lie algebra and $[Z,\sug(2)]=0$. If $[\sug(2),\ngo] \ne 0$ and $\ad Z$ is a positive multiple of the identity on each $\sug(2)$-irreducible subspace of $\ngo$, then $\lgo$ admits an inner product with negative Ricci curvature.
\end{theorem}

Another case where we can apply our method is when one starts with the non-compact dual of $\sug(m)$, $\slg(m, \RR)$. In this case, by Weyl's unitary trick, for each representation of $\sug(m)$ one gets a representation of $\slg(m, \RR)$. In Proposition \ref{Gengl} we show that $(\RR Z \oplus \slg(m, \RR)) \ltimes V$ admits an inner product with negative Ricci curvature for $m\ge 2$, where the space $V$ is the same as in the $\sug(m)$ case, the complex homogeneous polynomials in $m$ variables viewed as real. Although this result comes from a continuous argument, in each case one can actually make explicit the inner product on $\lgo$.

As a generalization of this we consider Lie algebras $(\ag \oplus \rg) \ltimes \ngo$ where $\ngo$ is nilpotent, $\ag$ is abelian and $\rg$ is semisimple of non-compact type, and obtain the following existence result.

\begin{theorem}\label{intronc}
Let $\ggo= (\ag \oplus \rg) \ltimes \ngo$ be a Lie algebra where  $\rg$ is a semisimple Lie algebra with no compact factors, $\ngo$ is nilpotent and $[\ag, \ag \oplus \rg]=0$. If in addition,
\begin{enumerate}
  \item  $\ngo$ admits an inner product such that $\ad A|_\ngo$ are normal operators for any $A \in \ag$,
  \item no $\ad A|_\ngo$ has all its eigenvalues purely imaginary,
  \item there exists an element $A_0$ in $\ag$ such that all the eigenvalues of $\ad A_0|_\ngo$ have positive real parts,
  \item  $\rg$ admits an inner product with orthogonal Cartan decomposition and with $\Ricci <0$,
\end{enumerate}
then $\ggo$ admits an inner product with negative Ricci curvature.
\end{theorem}

In \cite{DLM} it is proved that most of the non-compact simple Lie algebras admits an inner product satisfying (4). Therefore, starting from there, this theorem
gives a very large family of new examples.

It is worthwhile to remark that we provide in this paper many new examples of Lie groups admitting a Ricci negative metric which are not homeomorphic to any of the previously known examples in the literature.

\

{\bf Acknowledgements.} I wish to thank M. Jablonski for many useful comments and J. Lauret for very fruitful conversations on the topic of the paper. I am also very grateful to the referee for many suggestions which certainly improved the presentation and some results of the paper.

\section{Preliminaries and notation}

\subsection{Lie algebras}

We recall some background from \cite{u2} we will need along the paper. Let $\ggo=(\RR^m, \lb )$ be a Lie algebra of dimension $m$, that is, the underlying linear space of $\ggo$ is (identified with) $\RR^m$ and $\lb$ belongs to the space of Lie brackets $\mathcal{L}_m \subset \Lambda^2 (\RR^m)^* \otimes \RR^m$, defined as
$$
\mathcal{L}_m:=\{\mu:\RR^{m}\times\RR^{m}\to \RR^{m} : \mu\; \mbox{bilinear,
skew-symmetric and satisfies Jacobi}\}.
$$
$\mathcal{L}_m$ is also called the variety of Lie algebras of dimension $m$.
We consider the following action of $\Gl_m(\RR)$ on $\mathcal{L}_m:$
$$
(g \cdot\mu)(X,Y)= g \mu(g^{-1}X,g^{-1}Y), \qquad g \in \Gl_m(\RR),\; \mu\in \mathcal{L}_m,\; X,Y \in \ggo.
$$
Note that $(\RR^m, \mu)$ is isomorphic to $(\RR^m, g\cdot\mu)$ for any $g \in \Gl_m(\RR)$, though, $(\RR^m, \mu)$ is not isomorphic to $(\RR^m, \mu_o)$ for $\mu_o$ in the boundary of the orbit $\Gl_m(\RR)\cdot\mu$. Since $\mathcal{L}_m \subset  \Lambda^2 (\RR^m)^* \otimes \RR^m$  is defined by polynomials equations, any $\mu_0$ in the closure is also a Lie bracket. We will say that $\mu_o$ is a {\it degeneration} of $\mu$ or that $\mu$ degenerates to $\mu_o$ if $\mu_o \, \in \, \overline{\Gl_m(\RR)\cdot\mu}.$
Note that by continuity, many of the properties of $\mu_o$ are shared by $\mu.$ In particular if $(\RR^m, \mu_o)$ admits a metric with negative (or positive) sectional or Ricci curvature, so does $(\RR^m, \mu)$ (see \cite[Remark 6.2]{lowsol} or \cite[Proposition 1]{NN}).

 \begin{proposition}\label{limit}
  Suppose $\mu,\, \lambda \in \mathcal{L}_m$ and that $\lambda$ is in the closure of the orbit $\Gl_m(\RR)\cdot\mu.$ If the Lie algebra $(\RR^m, \lambda)$ admits an inner product of negative Ricci curvature, then so does the Lie algebra $(\RR^m, \mu).$
\end{proposition}

Moreover, if we fix an inner product on $\ggo=(\RR^m, \mu)$, or directly, an orthonormal basis,  then
the orbit $\Gl_m(\RR) \cdot \mu$  parameterizes, from a different point of
view, the set of all inner products on $\ggo$. Indeed,
\begin{equation}\label{chBasis}
(\ggo,g \cdot \mu,\ip) \mbox{ is isometric to }(\ggo,\mu,\la g\cdot,g\cdot\ra) \; \mbox{ for any }
g\in\Gl(\ggo).
\end{equation}

Let $(\ggo, \lb, \ip)$ be a metric Lie algebra  and $\Hm\in\ggo$ the only element such that $\la
\Hm,X\ra=\tr{\ad{X}}$ for any $X\in\ggo$, usually called the mean curvature vector, and let $B$ denote the symmetric map defined by
the Killing form of $(\ggo, \lb)$ (i.e. $\la BX,X\ra=\tr{(\ad{X})^2}).$ The Ricci operator of $(\ggo, \lb, \ip)$
is given by (see for instance \cite{B} Corollary 7.38):
\begin{equation}\label{ricci}
\Ricci=M-\unm B-S(\ad{\Hm}),
\end{equation}

\noindent where, $S(\ad{\Hm})=\unm(\ad{\Hm}+(\ad{\Hm})^t)$ is the symmetric part of $\ad{\Hm}$
and $M$ is the symmetric operator defined by
\begin{equation}\label{R}
\la MX,X\ra=-\unm\sum\la [X,X_i],X_j\ra^2 +\unc\sum\la
[X_i,X_j],X\ra^2,\qquad\forall X\in\ggo,
\end{equation}

\noindent where $\{ X_i\}$ is any orthonormal basis of $(\ggo, \ip)$. Note that if $\ggo$ is nilpotent, then $\Ricci=M$.

If $\ggo$ is a metric solvable Lie algebra and we consider an
orthogonal decomposition
\begin{equation}\label{solvdecomp}
  \ggo=\ag\oplus\ngo,
\end{equation}
where $\ngo$ is the nilradical of $\ggo$ (i.e. maximal nilpotent ideal), the expression of $\Ricci$ is much simpler when $\ag$ is abelian (see \cite{solvsolitons} (25)). Indeed, we get
\begin{equation}\label{ricsol}
\begin{array}{l}
\la\Ricci A,A\ra =  -\tr{S(\ad{A}|_{\ngo})^2}, \\ \\
\la\Ricci A,X\ra = -\unm\tr{(\ad{A}|_{\ngo})^t\ad{X}|_{\ngo}}, \\ \\
\la\Ricci X,X\ra = -\unm\sum\la [X,X_i],X_j\ra^2 +\unc\sum\la [X_i,X_j],X\ra^2 \\ \\
\qquad  \qquad \qquad +\unm\sum\la [\ad{A_i}|_{\ngo},(\ad {A_i}|_{\ngo})^t]X,X\ra -\la[\Hm,X],X\ra,
\end{array}
\end{equation}

\noindent for all $A\in\ag$ and $X\in\ngo$, where $\{ A_i\}$, $\{ X_i\}$, are any
orthonormal basis of $\ag$ and $\ngo$, respectively.
If in addition $\ad{A}$ are normal operators for all $A\in\ag$, then we get that $\tr{(\ad{A}|_{\ngo})^t\ad{X}|_{\ngo}}=0$
(see \cite[Lemma 4.7]{solvsolitons}) and therefore
\begin{equation}\label{ricsolnor}
\begin{array}{l}
\la\Ricci A,A\ra =  -\tr{S(\ad{A}|_{\ngo})^2}, \qquad
\la\Ricci A,X\ra = 0,\\ \\
\la\Ricci X,X\ra = -\unm\sum\la [X,X_i],X_j\ra^2 +\unc\sum\la [X_i,X_j],X\ra^2
-\la[\Hm,X],X\ra.
\end{array}
\end{equation}

Another useful tool, when we are dealing with the Ricci operator, specially in the nilpotent case, is the concept of a nice basis.

\begin{definition}\label{nicecond}
  Let $\ggo$ be a Lie algebra. We will say that the basis $\{X_1, \dots, X_n \}$ is {\it nice} if the structural constants given by $[X_i,X_j]=\sum c_{ij}^k X_k$ satisfy
$$
\begin{array}{@{\bullet\;\;}l}
  \text{for all } i,j \text{ there exists at most one } k \text{ such that } c_{ij}^k \ne 0, \\
  \text{for all } i,k  \text{ there exists at most one } j \text{ such that } c_{ij}^k \ne 0.
\end{array}
$$
\end{definition}
It is easy to see that when the basis is nice the operator $M$ is diagonal not only in such a basis but also in any rescaling (see (\ref{R})). If $\ggo$ is nilpotent (thus $M = \Ricci$), a much stronger result holds (see \cite{nicebasis}).

\subsection{Root decomposition}\label{root}

Let $\ggo$ be a complex semisimple Lie algebra and $B$ its Killing form. Let $\hg$ be a Cartan subalgebra, $\Delta$ the corresponding system of roots, $\Delta^+$ the positive ones with respect to some fixed order and let $\Pi$ be the corresponding set of simple roots. For each $\alpha \in \Delta$ it is possible to choose $X_\alpha \in \ggo_\alpha$, the corresponding root space, such that for all $\alpha, \beta \in \Delta$
\begin{equation}\label{sperootvect}
\begin{array}{l}
[X_{\alpha},  X_{-\alpha}] = H_\alpha, \qquad [H, X_{\alpha}]= \alpha(H) X_{\alpha}, \text{ for all } H \in \hg,\\ \\

[X_{\alpha},  X_{\beta}]=0, \,\,\text{ if } \alpha+\beta \ne 0, \, \alpha+\beta \notin \Delta, \quad

[X_{\alpha},  X_{\beta}]= N_{\alpha,\beta} X_{\alpha+\beta},\,\, \text{ if } \alpha+\beta \in \Delta,
\end{array}
\end{equation}
\noindent  where $N_{\alpha,\beta} = - N_{-\alpha,-\beta} \in \RR$. Here $H_\alpha$ denotes the dual vector associated to $\alpha$, i.e. $B(H,H_\alpha)= \alpha(H)$ for all $H \in \hg$ (see \cite{Hel} Theorem 5.5 or \cite{Kn} Theorem 6.6).

From this,
\begin{equation}\label{genu}
\begin{array}{l}
\ug = \displaystyle{\sum_{\alpha \in \Delta}} \RR \i H_\alpha + \sum_{\alpha \in \Delta} \RR (X_\alpha - X_{-\alpha}) + \sum_{\alpha \in \Delta} \RR \i(X_\alpha + X_{-\alpha}) \end{array}
\end{equation}
\noindent is a compact real from of $\ggo$. Moreover, if we denote $H^\alpha=\i H_\alpha$, $X^\alpha = (X_\alpha - X_{-\alpha})$ and $Y^\alpha = \i (X_\alpha + X_{-\alpha})$
we get
\begin{equation}\label{lbroot}
\begin{array}{l}
[H^\alpha, X^\beta] = c_{\alpha,\beta}  Y^\beta, \,\, [H^\alpha, Y^\beta]= - c_{\alpha,\beta} X^\beta,\\ \\

[X^\alpha, X^\beta]= N_{\alpha,\beta} X^{\alpha+\beta}-  N_{-\alpha,\beta} X^{-\alpha+\beta},\; \beta \ne \pm \alpha,\\ \\

[Y^\alpha, Y^\beta]= - N_{\alpha,\beta} X^{\alpha+\beta} -  N_{-\alpha,\beta} X^{-\alpha+\beta},\; \beta \ne \pm \alpha,\\ \\

[X^\alpha,  Y^\beta]= N_{\alpha,\beta} Y^{\alpha+\beta}-  N_{-\alpha,\beta} Y^{-\alpha+\beta},\;\beta \ne \pm \alpha, \\ \\

[X^\alpha, Y^\alpha]= 2 H^\alpha,
 \end{array}
\end{equation}

\noindent where $c_{\alpha,\beta}$ is a real number and $N_{\alpha,\beta}=0$ if $\alpha+\beta \notin \Delta$ (see  \cite{Kn} Theorem 6.11).
If we denote by $\kg = \Span\{X^\alpha, \alpha \in \Delta\}$ and $\pg= \Span\{ H_\alpha, \i Y^\alpha,\, \alpha \in \Delta \}$ then
$\ug = \kg \oplus \i \pg$ is a compact real form of $\ggo$ and $\ggo_0=\kg \oplus \pg$ is a Cartan decomposition of a non-compact real form of $\ggo$ (see \cite{Kn} pag. 360).
We say that $\ug= \kg \oplus \i \pg$ is the compact dual of $\ggo_0= \kg \oplus \pg$.
It worth to point out that the compact real form of a semisimple complex Lie algebra is unique up to isomorphism (see \cite[Corollary 6.20]{Kn}).

\section{Ricci negative inner products}

We will show in what follows that for a compact (real) semisimple Lie algebra $\ug$ and a real representations of $\ug$, $(V,\pi)$, satisfying some conditions, the semidirect product  $\lgo= (\RR Z \oplus \ug)\ltimes V$, where $\RR Z \oplus \ug$ is a central extension of $\ug$, admits an inner product with negatively defined Ricci operator.  We will first show that $\lgo$ degenerates into a solvable Lie algebra and then we will prove that this limit admits an inner product with $\Ricci < 0$ and hence so does the starting Lie algebra.
Note that given a compact semisimple Lie algebra $\ug$, by the uniqueness up to isomorphism of a real compact form, we will often denote by $\Delta$, $\Delta^+$ and $\Pi$ the corresponding root data of its complexification and a basis of $\ug$ as in (\ref{genu}), with no further comments.

\begin{lemma}\label{lema0} Let $\ug$ be a compact semisimple Lie algebra and let $(V,\pi)$ be a finite dimensional, real representation of $\ug$ such that $V=V_1 \oplus V_2$ where $V_1$ an $V_2$ are $H^\alpha$-invariant subspaces  and
$\pi(X^\alpha)(V_1) \subset V_2$, and $\pi(Y^\alpha)(V_1) \subset V_2$, for any $\alpha \in \Delta^+$.
Then the Lie algebra $\lgo= (\RR Z \oplus \ug)\ltimes V$ where $\ad Z|_\ug = 0$ and $\ad Z|_V = Id$ degenerates in a solvable Lie algebra $\lgo_\infty =(\RR^l, \mu)$, where $l= \dim \lgo=\dim \ug + 1 + \dim V$.
\end{lemma}

\begin{proof}
Let $\ug$ be a compact semisimple Lie algebra and $(V,\pi)$ a finite dimensional (real) representation of $\ug$. Let  $\lgo= (\RR Z \oplus \ug)\ltimes V$  be the Lie algebra where $\ad Z|_\ug = 0$ and $\ad Z|_V = Id$.
 If $\ggo$ is the complexification of $\ug$ let us consider the basis of $\ug$ as in (\ref{genu}) and from there, let as fix a basis of $\lgo$, $\mathcal{B}$ given by
\begin{equation}\label{B}
  \mathcal{B}=\{Z, H^\alpha, X^\beta, Y^\beta, \, \alpha \in \Pi, \beta \in \Delta^+\} \cup \mathcal{B}_1,
\end{equation}
\noindent where $\mathcal{B}_1 = \{v_i:\, 1 \le i \le m \}$ is any basis of $V$ with $v_i \in V_1$ for $1 \le i \le r = \dim V_1$ and $v_i \in V_2$ for $r <  i \le m=\dim V$.

For each $t>0$ define $\phi_t \in \glg(\lgo)$ such that
\begin{equation}\label{degU}
\begin{array}{l}
{\phi_t}(Z) = Z, \qquad {\phi_t}(H^\alpha) = H^\alpha, \; \alpha \in \Pi, \\ \\
{\phi_t}(X^\alpha) = t X^\alpha, \; {\phi_t}( Y^\alpha) = t Y^\alpha, \quad \alpha \in \Delta^+,  \\ \\
{\phi_t}(v_i) = \left\{ \begin{array}{ll}
                                     \tfrac{t}{\rho} \, v_i & \hbox{if } i \le r, \\ \\
                                     t^2\,v_i, & \hbox{if } i > r,
                                   \end{array}
                                 \right.
\end{array}
\end{equation}
\noindent where $\rho \in \RR$, $\rho\ne 0$, is fixed.
It is not hard to check that $\lb_t = \phi_t.\lb$ is given by
\begin{equation}\label{degU}
  \begin{array}{l}
[H^\alpha,X^\beta]_t = [H^\alpha,X^\beta], \;\; [ H^\alpha, Y^\beta]_t=[H^\alpha, Y^\beta], \;\; \forall\,\alpha \in \Pi, \beta \in \Delta^+ \\ \\

[X^\alpha, Y^\beta]_t=\frac{1}{t^\epsilon}[X^\alpha, Y^\beta],\; \epsilon=1 \text{ if } \beta \ne \alpha, \epsilon=2 \text{ if } \beta=\alpha,  \\ \\

[Y^\alpha,Y^\beta]_t=\frac{1}{t}[Y^\alpha,Y^\beta],\,\, [X^\alpha,X^\beta]_t=\frac{1}{t}[X^\alpha,X^\beta], \;\beta \ne \alpha,  \\ \\

[Z, v_i]_t= v_i,\;  \;\; [H^\alpha,  v_i]_t= [H^\alpha, v_i],\; \quad \forall\,i,\, \alpha \in \Pi,  \\ \\

[X^\alpha, v_i]_t=\rho\, [X^\alpha, v_i],\; \;\; [Y^\alpha, v_i]_t= \rho\, [Y^\alpha, v_i],\;\;  \forall\, \alpha \in \Delta^+,\, i \le r, \\ \\

[X^\alpha, v_i]_t=\frac{1}{\rho\, t^2}\displaystyle{\sum_{j \le r}} c(X^\alpha, v_i)^j v_j + \frac{1}{t}\displaystyle{\sum_{j>r}} c(X^\alpha, v_i)^j v_j, \forall\, \alpha \in \Delta^+,\, i > r, \\

[Y^\alpha, v_i]_t=\frac{1}{\rho\, t^2}\displaystyle{\sum_{j \le r}} c(Y^\alpha, v_i)^j v_j + \frac{1}{t}\displaystyle{\sum_{j>r}} c(Y^\alpha, v_i)^j v_j,, \forall\, \alpha \in \Delta^+,\, i > r,
\end{array}
\end{equation}
\noindent where $c(Y^\alpha, v_i)^j$ are the structure coefficients, i.e. $[Y^\alpha, v_i] = \displaystyle{\sum_{j = 1}^m} c(Y^\alpha, v_i)^j v_j$.

Therefore, $\mu=\displaystyle{\lim_{t \to \infty}} \lb_t= \displaystyle{\lim_{t \to \infty}} \phi_t.\lb$ is well defined for any $\rho$ and the non-trivial brackets are given by
\begin{equation}\label{limitUg}
   \begin{array}{l}
\mu(H^\alpha,X^\beta)=[H^\alpha,X^\beta], \;\; \mu(H^\alpha,Y^\beta)=[H^\alpha,Y^\beta],   \\ \\
 \mu(Z, v_i)= v_i,\;  \;\; \mu(H^\alpha, v_i)= [H^\alpha, v_i ]= \pi(H^\alpha) v_i,\; \quad  \\ \\
\mu(X^\alpha, v_i)= \rho [X^\alpha, v_i]=\rho \pi(X^\alpha)v_i,\; \mu( Y^\alpha, v_i)= \rho [Y^\alpha, v_i]=\rho \pi(Y^\alpha)v_i ,\;\;  i \le r.
\end{array}
\end{equation}

It is easy to see that $\lgo_\infty = (\RR^{l}, \mu),$  is a solvable Lie algebra with nilradical
\begin{equation}\label{ngogen}
  \ngo = \Span \{X^\alpha,Y^\alpha, v_i, \;\, \alpha \in \Delta^+, 1\le i \le m\},
\end{equation}
\noindent as claimed.
\end{proof}

It is proved in \cite[Theorem 2]{NN} that if $\sg$ is a solvable Lie algebra, $\ngo$ its nilradical and $\zg$ is
the center of $\ngo$, then a necessary condition for $\sg$ to admit an inner product of negative Ricci curvature, is the existence of $Y \in \sg$
such that $\tr \ad Y > 0$ and all the eigenvalues of the restriction of the operator $\ad Y$ to
$\zg$ have a positive real part.  On the other hand, if there exists $Y \in \sg$ such that all the eigenvalues of the restriction of $\ad Y$ to $\ngo$ have
positive real part, then $\sg$ admits an inner product with $\Ricci < 0$.

If we denote by $\ngo_\ug =\Span \{X^\alpha,Y^\alpha,\;\, \alpha \in \Delta^+\}$  then the nilradical of $l_\infty$ given in (\ref{ngogen}), decomposes as $\ngo=\ngo_\ug \oplus V$.
It is clear from the definition of the bracket $\mu$ (see \ref{limitUg}) that $\ngo$ is a two step nilpotent Lie algebra. It is abelian if and only if $\pi(X)|_{V_1}=0$ for every $X \in \ngo_\ug$. In the case when for at least one $X \in \ngo_\ug$, $\pi(X)$ is not trivial on $V_1$ then  $\lgo_\infty$ satisfies the first condition of \cite[Theorem 2]{NN} but not the second one.

\begin{remark}
We note that in fact $\phi_t$ and $\lgo_\infty$ depend on $\rho$. Nevertheless, all the limits we get using different values of $\rho$ are isomorphic as Lie algebras.
\end{remark}

\begin{theorem}\label{theTheo}
Let $\ug$ be a compact semisimple Lie algebra, $(V,\pi)$ a representation of $\ug$ with $V=V_1\oplus V_2$ and $\lgo$ as in Lemma \ref{lema0}. If in addition, there exist an inner product on $V$ such that $V_1$ is orthogonal to $V_2$ and
 \begin{itemize}
   \item[i)]  $\pi(H^\alpha)$ is a skew-symmetric operator of $V$ for any $\alpha \in \Pi$,
   \item[ii)] $\pi(X)_{|_{V_1}}$ is not trivial for every $X=X^\alpha, Y^\alpha$, $\alpha \in \Delta^+$ and
   \item[iii)] $\tr \pi(Y)_{|_{V_1}}^t \pi(X)_{|_{V_1}} = 0$ whenever $X \ne Y$ are elements of $\{X^\alpha, Y^\alpha,\, \alpha \in \Delta^+\}$,
 \end{itemize}
  then $\lgo$ admits an inner product with negative Ricci curvature.
\end{theorem}

\begin{remark}
Since $\ug$ is a compact Lie algebra, there exists an inner product on $V$ such that $\ug$ acts by skew-symmetric operators. Also note that by ii) the center of the nilradical of $\lgo_\infty$ is not abelian and therefore $\lgo_\infty$ is a solvable Lie algebra that satisfies the first condition of \cite[Theorem 2]{NN} but not the second one.
\end{remark}

\begin{proof}

Since $\ug$ and $(V,\pi)$ satisfy the hypotheses of Lemma \ref{lema0}, the Lie algebra $\lgo = (\RR Z \oplus \ug)\ltimes V = (\RR^l,\lb)$ degenerates in the solvable Lie algebra $\lgo_\infty=(\RR^l,\mu)$, where $\mu$ is given by (\ref{limitUg}). We will show that $\lgo_\infty$ admits an inner product such that the corresponding Ricci operator is negative definite and therefore so does $\lgo$ by Proposition \ref{limit}.

 Let $\mathcal{B}$ be the basis of $\lgo_\infty$ given in (\ref{B})  and let $\ip$ be the inner product on $\lgo_\infty$ such that $\ip_{|_V}$ is as in the statement and $\mathcal{B}$ is an orthonormal basis of $\lgo_\infty$.

Note that $\lgo_\infty = \ag \oplus \ngo$ as in (\ref{solvdecomp}) where $\ag= \Span \{Z, H^\alpha,\, \alpha \in \Pi \}$ is abelian and for each $\alpha \in \Pi$,  $\ad_{\mu}(H^\alpha)_{|\ngo}$ is a skew-symmetric operator (see (\ref{lbroot})).

It is easy to check that the mean curvature vector is  $\Hm = (\dim V) Z = m \, Z$  and since $\ag$ is acting by normal operators on $\ngo$, $\la \Ricci_{\mu} \ag, \ngo \ra = 0$ (see (\ref{ricsolnor})).

Recall from (\ref{ricsolnor}) that for $X,Y \in \ngo$,  $\la \Ricci_{\mu} X , Y  \ra $ is given by
{\small \begin{equation}\label{ricxy}  -\unm\sum\la [X,X_i],X_j\ra \la [Y,X_i],X_j\ra  + \unc\sum\la [X_i,X_j],X\ra \la [X_i,X_j],Y\ra
-\la[S(\Hm),X],Y\ra,\end{equation}}

\noindent where in this case $\{X_i\}$ is the orthonormal basis of $\ngo$
$$
\{ X^\alpha, Y^\alpha,v_i: \, \, \alpha \in \Delta^+, 1 \le i \le m\}.
$$
It is clear that the second term of (\ref{ricxy}) can only be nontrivial for $X,Y \in V_2$ and the third one is zero unless $X = Y \in V$. We thus get for
 $X, Y \in \{X^\alpha, Y^\alpha,\, \alpha \in \Delta^+\}$,
 $$ \begin{array}{l}
 \la \Ricci_{\mu} X , Y  \ra  =  -\tfrac{\rho^2}{2}\displaystyle{\sum_{i \le r < j }} \la \pi(X) v_i,v_j\ra \la \pi(Y) v_i,v_j\ra, \\
 \qquad = -\tfrac{\rho^2}{2}\displaystyle{\sum_{ i \le r}} \la \pi(X) v_i, \pi(Y) v_i\ra
 = -\tfrac{\rho^2}{2} \tr \pi(Y)_{|_{V_1}}^t \pi(X)_{|_{V_1}}.
  \end{array}$$

By iii) this is zero for $X \ne Y$ and it is negative for $X = Y$ (see ii)).

It also follows from (\ref{ricxy}) that $\la \Ricci_{\mu} V_1 , V_2 \ra =0$ and
for $X=v_k,\,Y=v_j$, $k,j \le r$,

$$ \begin{array}{l}
\la \Ricci_{\mu} v_k , v_j  \ra  =   -\tfrac{\rho^2}{2} \displaystyle{ \sum_{\substack{ X=X^\alpha, Y^\alpha \\ l > r }}} \la \pi(X) v_k,v_l\ra \la \pi(X) v_j,v_l\ra - m \la [Z, v_k], v_j \ra, \\
\quad  =  -\tfrac{\rho^2}{2} \displaystyle{ \sum_{ X=X^\alpha, Y^\alpha }} \la \pi(X) v_k, \pi(X) v_j \ra - m \, \delta_{k,j}.
 \end{array}$$

In the same way, for $X=v_k,\,Y=v_j$, $k,j > r$,
$$ \begin{array}{l}
\la \Ricci_{\mu} v_k , v_j  \ra  =  \tfrac{\rho^2}{4} \displaystyle{ \sum_{\substack{ X=X^\alpha, Y^\alpha \\ l \le r }}} \la \pi(X) v_l,v_k\ra \la \pi(X) v_l,v_j\ra - m\; \delta_{k,j}  \\
\qquad =  \tfrac{\rho^2}{4} \displaystyle{ \sum_{ X=X^\alpha, Y^\alpha }} \la \pi(X)_{|_{V_1}}^t v_k, \pi(X)_{|_{V_1}}^t v_j\ra - m \, \delta_{k,j} .
 \end{array}$$

We therefore get that $\Ricci_\mu$ restricted to $V=V_1 \oplus V_2$ is given by

{\small $$ \left[\begin{array}{c:c} -\tfrac{\rho^2}{2} \displaystyle{\sum_{X=X^\alpha, Y^\alpha}} \pi(X)_{|_{V_1}}^t \pi(X)_{|_{V_1}} - m\,  \Id_{V_1} & 0 \\ \hdashline 0 & \tfrac{\rho^2}{4} \displaystyle{\sum_{X=X^\alpha, Y^\alpha}} \pi(X)_{|_{V_1}} \pi(X)_{|_{V_1}}^t - m \, \Id_{V_2}   \end{array} \right].
$$}

Since $\pi(X)_{|_{V_1}}^t \pi(X)_{|_{V_1}}$ and $\pi(X)_{|_{V_1}} \pi(X)_{|_{V_1}}^t$ are symmetric operators of $V_1$ and $V_2$ respectively for any $X \in \{X^\alpha, Y^\alpha,\, \alpha \in \Delta^+\}$ and $-m \Id$ is negative definite, there exists a small $\rho$ so that the finite sums
$$-\tfrac{\rho^2}{2} \displaystyle{\sum_{X=X^\alpha, Y^\alpha}} \pi(X)_{|_{V_1}}^t \pi(X)_{|_{V_1}} - m \, \Id_{V_1}, \quad
\tfrac{\rho^2}{4} \displaystyle{\sum_{X=X^\alpha, Y^\alpha}} \pi(X)_{|_{V_1}} \pi(X)_{|_{V_1}}^t - m \, \Id_{V_2}$$

\noindent are both negative definite.

Concerning $\ag$ we have that
$$ \la \Ricci_{\mu} Z,Z \ra = - m, \qquad \la \Ricci_{\mu} H^\alpha,H^\alpha \ra = 0, \, \forall \alpha \in \Pi.$$

Therefore, $\Ricci_\mu$ is a non-positive operator. To get negative, we only need that no ${\ad H^\alpha}_{|_\ngo}$ acts as a skew-symmetric operators. In order to get that, note that one can slightly perturb the inner product on $V_1$ so that no $H^\alpha$ acts as a skew-symmetric operator on $V_1$ for any $\alpha \in \Pi$ and still have $\la \Ricci(\ag),\ngo \ra=0$. Indeed, we only have to check that
$$\la\Ricci A,X\ra = -\unm\tr{(\ad{A}_{|_{\ngo}})^t\ad{X}_{|_{\ngo}}}=0, $$
\noindent for $A=Z, H^\alpha$ and $X= X^\alpha, Y^\alpha$ or $v_i$. It is clear that nothing changes for $Z$ since it $\ad Z_{|_V} = \Id$. For $H^\alpha$ note that on the one hand,  $\ad H^\alpha$ leaves the orthogonal subspaces $\ngo_\ug$, $V_1$ and $V_2$ invariant and on the other hand,
{\small \begin{equation}\label{adnu}
  \begin{array}{l}
  {\ad X^\alpha}_{|_\ngo}=\left[\begin{array}{cc:cc} 0 &  &  & \\ &0&& \\ \hdashline  &&0&0  \\ &&C(X^\alpha)&0   \end{array} \right], \qquad
{\ad Y^\alpha}_{|_\ngo} = \left[\begin{array}{cc:cc} 0 &  &  & \\ &0&& \\ \hdashline  &&0&0  \\ &&C(Y^\alpha)&0   \end{array} \right], \\ \\
 {\ad v_i}_{|_\ngo}= \left[\begin{array}{cc:cc} 0 &  &  & \\ &0&& \\ \hdashline  &&0&  \\ A& B&&0   \end{array} \right],\,\,  i \le r, \quad
{\ad v_i}_{|_\ngo} = 0,\,\,   i > r,
\end{array}
\end{equation}}
\noindent where the blocks correspond to the decomposition $\ngo = \ngo_\ug \oplus V$. Therefore, it is easy to see that for any small
perturbation of the inner product on $V_1$, $\tr{(\ad{H^\alpha}|_{\ngo})^t\ad{X}|_{\ngo}}=0$ for any $X$ in (\ref{adnu}).
Hence we can take a small perturbation of the inner product on $V_1$ such that no ${\ad H^\alpha}_{|_{V_1}}$ acts skew-symmetrically
(see \cite[Lemma 2]{NN})
and therefore the corresponding Ricci operator is negative definite as we wanted to show.
\end{proof}

\section{Explicit Examples}

In this section, we give explicit examples of representations for which the conditions in Theorem \ref{theTheo} hold, so we get explicit examples of negative Ricci curved Lie groups. Indeed we will show that for a representation of $\sug(m)$, $\sog(m)$ and $\spg(m)$ in some space of homogeneous polynomials, Theorem \ref{theTheo} can be applied.
In all the cases we will exhibit, as matrices, the realization of $\ug$ as in (\ref{genu}) to be able to get explicit formulas for the action to finally
show that in fact all the hypotheses of Theorem \ref{theTheo} are satisfied. All the cases behave in a very similar way, so we will give more details in the first one.

The study of whether the conditions of Theorem \ref{theTheo} hold for a more general representation requires a deep study of the systems of roots and weights which will be worked out in a forthcoming paper.

\subsection{Polynomial representations of $\sug(m)$}\label{repre}
For each $n \ge 2$ let $(\pi_n, W_n)$ be the representation of $\SU(m)$  where
$W_n =\mathcal{P}_{n}(\mathbb{C}^m)$ is the space of homogeneous polynomials in $m$
variables of degree $n$ seen as a real vector space with action given by
$$ (\pi_n(g) P)(z_1,\dots, z_m )=P(g^{-1}\left[ \begin{smallmatrix} z_1 \\ \vdots \\ z_m \end{smallmatrix}\right]).$$
This gives us, by differentiation, a representation of the Lie algebra $\sug(m)$ that will also be denoted by $(\pi_n, W_n)$.
Note that these representations are not irreducible in general.

Let us start from $\ggo^\CC = \slg(m,\CC)$, a complex simple Lie algebra of type $A_{m-1}$ and choose the Cartan subalgebra as the diagonal matrices in $\ggo^\CC$. To set the corresponding root system, let
$E_{i,j}$ denote the $m \times m$ matrix with zero entries except for the $(i,j)$ which is $1.$
Hence the roots of $\ggo^\CC$ are
\begin{equation}\label{roots}
  \alpha_{i,j}(H) = e_i(H)-e_j(H), \qquad 1\le i\ne j \le m,
\end{equation}
\noindent  where if $H = \displaystyle{\sum_{l=1}^{m}} h_l E_{l,l}$,  $e_k(H)=h_k$ and the corresponding root vectors are $X_{\alpha_{ij}}=E_{i,j}$ (see \cite{Hel} pp. 187).
We choose an order so that the positive roots are $\{\alpha_{i,j}, \, i \le j\}$ and the simple ones are  $\{\alpha_{l,l+1}, \, 1 \le l \le m-1\}.$
Hence, in the notation of Section \ref{root},
\begin{equation}\label{formCom}
   \begin{array}{l}
   H^{\alpha_{l,l+1}} = \i H_{\alpha_{l,l+1}} = \i (E_{l,l}- E_{l+1,l+1}), \quad l=1,\dots m-1, \\ \\
    X^{\alpha_{i,j}}= X_{\alpha_{ij}} - X_{-\alpha_{ij}} = E_{i,j} - E_{j,i}, \quad  1\le i < j \le m,\\ \\
 Y^{\alpha_{i,j}}= \i(X_{\alpha_{ij}} - X_{-\alpha_{ij}}) = \i(E_{i,j} + E_{j,i}), \quad  1\le i < j \le m.
\end{array}
\end{equation}

In this case (\ref{genu}) gives us $\ug = \sug(m)$.
Let us fix a basis of $W_n$,
\begin{equation}\label{basewm}
\mathcal{B}_1= \{p_{j_1,\dots,j_m},\; \i p_{j_1,\dots,j_m}, \quad j_i \in \NN_0,\;\; j_1+\dots+j_m=n \},
\end{equation}
where $p_{j_1,\dots,j_m} = z_1^{j_1}\dots z_m^{j_m} \in \mathcal{P}_{n}(\CC^m)$. Note that dimension of $W_n$ is $d=2\binom{n+m-1}{m-1}$.
Concerning the action, to get explicit formulas we use the fact that the algebra is acting by derivations and
{\small $$ \begin{array}{l}
 H^{\alpha_{l,l+1}}\cdot z_k= \left\{ \begin{array}{ll}
             -\i z_k, &  k=l,  \\
             \i z_k & k=l+1, \\
             0, & k \ne l,l+1,
               \end{array}  \right.
 X^{\alpha_{i,j}}\cdot z_k= \left\{ \begin{array}{ll}
             - z_j, & k=i, \\
              z_i, & k=j, \\
             0, & k \ne i,j,
           \end{array} \right.
Y^{\alpha_{i,j}}\cdot z_k= \left\{ \begin{array}{ll}
             -\i z_j, & k=i, \\
              -\i z_i, & k=j, \\
             0, & k \ne i,j.
           \end{array} \right.
\end{array}
$$}
In this way, we get for example that the monomials in $\mathcal{B}_1$ are weight vectors for the action of $\ggo^{\CC}$ and for $s=1,\i$
$$ X^{\alpha_{i,j}}\cdot s z_k^n = \left\{ \begin{array}{ll} - n\, s z_i^{n-1}z_j,& k=i,\\ n\, s z_j^{n-1}z_i,& k=j, \\ 0, & k\ne i,j. \end{array} \right.
$$
Let us take $V_1$ the subspace of $W_n$ generated by
\begin{equation}\label{DefS}
  \mathcal{S}=\{ z_k^n, \i z_k^n,\;\; k=1,\dots m\} \subset \mathcal{B}_1,
\end{equation}
and $V_2$ the subspace generated by $\mathcal{B}_1 \smallsetminus \mathcal{S}$.
It is easy to see that the inner product that makes $\mathcal{B}_1$ an orthonormal basis, satisfies the hypotheses of Theorem \ref{theTheo}. In fact, for every $l$, $H^{\alpha_{l,l+1}}$ leaves $V_1$ and $V_2$ invariant and acts as a skew-symmetric operator. It is also clear from the formulas of the action that for every $p \in \mathcal{S}$,  $X^{\alpha_{i,j}}\cdot p$ and $Y_{i,j}\cdot p$ are elements of
 $\Span (\mathcal{B}_1 \smallsetminus \mathcal{S})$. We note that here we are using that $n \ge 2$.

Hence, to apply Theorem \ref{theTheo} it only remains to show that
$\tr \pi_n(Y)_{|_{V_1}}^t \pi_n(X)_{|_{V_1}} = 0$ whenever $X \ne Y$ are elements in $\{X^{\alpha_{i,j}}, Y^{\alpha_{i,j}},\, 1 \le i < j \le m \}$. To see this, first note that
{\small  \begin{equation}\label{ricN}
  \tr \pi_n(Y)_{|_{V_1}}^t \pi_n(X)_{|_{V_1}} = \displaystyle{\sum_{\substack {1\le k \le m \\ p \in \mathcal{B}_1 \smallsetminus \mathcal{S}}}} \la \pi_n(X)z_k^n , p \ra \la \pi_n(Y) z_k^n, p\ra +  \la \pi_n(X) \i z_k^n , p \ra \la \pi_n(Y) \i z_k^n, p\ra,
  \end{equation}}
\noindent and using explicitly the action, it is easy to show that every term is zero.  
Indeed, if for example $X=X^\alpha$, $Y=Y^\beta$, $\alpha, \beta \in \Delta^+$, then  $\pi_n(X)$ leaves invariant the subspaces
$$\Span \{ p_{j_1,\dots,j_m},  j_1+\dots+j_m=n \},\, \text{ and }\, \Span \{\i p_{j_1,\dots,j_m}, \;\; j_1+\dots+j_m=n \}$$
\noindent  and $\pi_n(Y)$ interchanges them. For $X=X^\alpha$, $Y=X^\beta$, it is enough to note that
$$\la \pi_n(X^{\alpha_{i,j}}) s z_r^n, \pi_n(X^{\alpha_{k,l}} ) s z_r^n \ra =0, \text{ when } (i,j) \ne (k,l), s=1,\i.$$
We can therefore apply Theorem \ref{theTheo} to get:

\begin{theorem}
  Let $(W_n,\pi_n)$ be the standard real representation of $\sug(m)$ on the space of complex homogeneous polynomials of degree $n$ in $m$ variables $\mathcal{P}_{n}(\CC^m)$ extended
  to $\ug(m)$ by letting the center act as multiples of the identity. Hence the Lie algebra $\ug(m)\ltimes W_n$ admits an inner product with negative Ricci curvature for all $n,m \ge 2$.
\end{theorem}

\begin{remark}
As was noticed in \cite{u2} for $\sug(2)$, the case when the representation is $\CC^n$, i.e. the case when $n=1$, must be studied separately since the action is different and the degeneration given in (\ref{degU}) leads to a solvable Lie algebra with an abelian nilradical. It is shown in \cite{u2}, Lemma 3.4, that this problem can be solved for  $\sug(2)$.
\end{remark}

\begin{remark}
  In this case, a straightforward calculation shows that the Ricci operator restricted to $V=W_n$ is diagonal. Also, the inner product perturbation can be found explicitly by considering a rescaling of the vectors in  $\mathcal{S}$.
\end{remark}

\subsection{Polynomial representations of $\sog(m)$}

Let $(W_n, \pi_n)$ be the standard representation of $\sog(m)$ on the space of complex-valued homogeneous polynomials of degree $n$ on $\RR^m$ derived from the standard action of the group $\SO(m)$.

In \cite[Chap.\ IV,\S 5, Examples 1,2]{Kn2} it is shown that if $(x_1, \dots, x_m) \in \RR^m$, it is convenient to see these polynomials as powers of
{\small \begin{equation}\label{zetas}
\begin{array}{l}
  z_1=x_1+\i x_2, \;z_2=x_1 - \i x_2, \dots, z_{m-1}= x_{m-1}+\i x_{m}, \,z_{m}= x_{m-1} -\i x_{m},  \quad m \, \text{is even}\\ \\
  z_1=x_1+ \i x_2, \;z_2=x_1- \i x_2, \dots, z_{m-1}= x_{m-2}-\i x_{m-1}, \,z_{m}= x_{m},  \quad m \, \text{is odd}
\end{array}
\end{equation}}

since the weight vectors are
$$
\begin{array}{l}
(x_1+\i x_2)^{k_1}(x_1-\i x_2)^{r_1}\dots (x_{2l-1} -\i x_{2l})^{r_l}, \, \sum k_i+ \sum r_i = n,\;\;  \text{for} m=2l,\\ \\
  (x_1+\i x_2)^{k_1}(x_1-\i x_2)^{r_1}\dots (x_{2l-1} -\i x_{2l})^{r_l}x_{2l+1}^{k_0},\, \sum k_i+ \sum r_i = n,\;\;  \text{for} m=2l+1.
\end{array}
$$

Recall that to get a real representation we have to consider powers of $z_1, \i z_1, \dots, z_m, \i z_m$ and hence $d=\dim W_n = 2\binom{n+m-1}{n}$.

Note that if $m=2l+1$, $\sog(m)$ is a compact real form of the $B_l$-type complex Lie algebra $\sog(2l+1, \CC)$ and for $m=2l$ the corresponding type is $D_l$. Therefore, we will study these cases separately. Also recall that we have the following isomorphism $\sog(3)\simeq \sug(2)$ and therefore we may assume that $l \ge 2$.

We are using the following notation for the elements of $\sog(2l+1, \CC)$ as $(l+1) \times (l+1)$-block matrices where the first $l$ blocks are $2 \times 2$ and there is one more row and column:
\small{\begin{equation}\label{notmat}
 X = \left[\begin{array}{ccc:c} A_{1,1} &  \dots & A_{1,l} &  A_{1,l+1} \\
 & \ddots && \vdots \\ A_{l,1} & \dots& A_{l,l}& A_{l,l+1} \\ \hdashline A_{l+1,1} & \dots & A_{l+1,l}& 0
 \end{array} \right].
 \end{equation}}
That is $A_{i,j}$ is a $2 \times 2$ matrix if $i,j \le l$ and $A_{i,l+1}$ is a column matrix with 2 rows. Recall that $A_{j,i}=-A_{i,j}^t$.
Using \cite[Example 2, pp. 63]{Kn2}, we can choose as Cartan subalgebra
$$ \hg=\left\{ H \in \sog(2l+1, \CC): \, H= \left[ \begin{smallmatrix} \begin{array}{ccc:c}
    A_{1,1} & & & \\ &\ddots & & \\ && A_{l,l} & \\ \hdashline &&& 0 \end{array}
  \end{smallmatrix} \right], \, \text{ where } A_{i,i}=\left[ \begin{smallmatrix}  0 & \i h_i\\- \i h_i & 0
\end{smallmatrix}\right]\right\}.$$
For each $1 \le i \le m$, let $e_i \in \hg^*$ be defined by
$e_i(H)=h_i$ for any $H \in \hg$ as above. Then the system of roots is given by
$$\Delta =\{\pm e_i \pm e_j,\, \, i\ne j\} \cup \{\pm e_k \}.$$
We choose an order so that the positive roots are
$$\Delta^+ =\{\alpha_{i,j}^\pm=e_i \pm e_j,\, \, i< j\} \cup \{\alpha_k=e_k \}, \text{ and let } \Pi=\{\alpha_{i,i+1}^-, e_l,\,\, 1\le i\le l-1\},$$
\noindent be the simple ones.

The corresponding roots vectors $X_{\pm \alpha_{ij}^\pm}$, for $\pm \alpha_{i,j}^\pm$ $i<j$, have all its block-entries $0$ except for $A_{i,j}$ and $A_{j,i}$.
We obtain therefore a basis of $\sog(m)$ as in (\ref{genu}) given explicitly by
$$
\{H^{\alpha_{i,i+1}^-}, H^{e_l},  X^{\alpha^\pm_{k,j}}, X^{\alpha_r}, Y^{\alpha^\pm_{k,j}}, Y^{\alpha_r}, \; i,r \le l, k<j \le l\}.
$$
\small{ $$ \begin{array}{l}
H^{\alpha_{i,i+1}^-}= \left[\begin{smallmatrix} \ddots & & &   \\ & A_{i,i}& & \\ & & -A_{i+1,i+1}& \\ &&& \ddots \end{smallmatrix}\right],\, i \le l-1, \quad H^{e_l}= \left[\begin{smallmatrix} \ddots & &    \\ & A_{l,l}&  \\  &&& 0 \end{smallmatrix}\right],\, \text{ where }\,
A_{r,r}=\left[ \begin{smallmatrix} & -1\\1 &
\end{smallmatrix}\right],\\
X^{\alpha^\pm_{k,j}} = \left[\begin{smallmatrix} \ddots & & & &  \\ & & A_{k,j}^\pm& & \\ & A_{j,k}^\pm &&& \\ &&& \ddots \end{smallmatrix}\right], \, A_{k,j}^-=\left[ \begin{smallmatrix}2 & \\& 2
\end{smallmatrix}\right],\, A_{k,j}^+=\left[ \begin{smallmatrix} 2 & \\ & -2 \end{smallmatrix}\right],\,  k < j \ne l+1, \\
Y^{\alpha^\pm_{k,j}} =  \left[\begin{smallmatrix} \ddots & & & &  \\ & & A_{k,j}^\pm& & \\ & A_{j,k}^\pm &&& \\ &&& \ddots \end{smallmatrix}\right],\,  A_{k,j}^-=\left[ \begin{smallmatrix} & -2 \\ 2 &
\end{smallmatrix}\right],\, A_{k,j}^+=\left[ \begin{smallmatrix}  & 2 \\2 &  \end{smallmatrix}\right],\,  k < j \ne l+1, \\
 X^{\alpha_r} = \left[\begin{smallmatrix} \ddots & & & \vdots   \\ & & & A_{r,l+1} \\ & && \vdots \\ \dots & A_{l+1,r} && \end{smallmatrix}\right],
 A_{r,l+1}=\left[ \begin{smallmatrix} 2\\0 \end{smallmatrix}\right], \,
  Y^{\alpha_r} = \left[\begin{smallmatrix} \ddots & & & \vdots   \\ & & & A_{r,l+1} \\ & && \vdots \\ \dots & A_{l+1,r} && \end{smallmatrix}\right],
A_{r,l+1}=\left[ \begin{smallmatrix} 0\\2 \end{smallmatrix}\right], 1\le r\le l.
 \end{array}$$}

Straightforward calculations show that the non-trivial action of $\sog(2l+1)$ on $W_n$  can be obtained from
{\small $$ \begin{array}{l}
H^{\alpha_{i,i+1}^-}\cdot z_r= \left\{ \begin{array}{ll}
             -\i z_r, &  r=2i-1,2i+2,  \\
             \i z_r & r=2i,2i+1,
                  \end{array}  \right. i<l, \quad
H^{e_l}\cdot z_r= \left\{ \begin{array}{ll}
             -\i z_r, &  r=2l-1,  \\
             \i z_r & r=2l,
               \end{array}  \right. \\ \\
X^{\alpha^+_{k,j}}\cdot z_r= \left\{ \begin{array}{ll}
             -2 z_{2j}, & r=2k-1, \\
             -2 z_{2j-1}, & r=2k, \\
             2 z_{2k}, & r=2j-1, \\
             2 z_{2k-1}, & r=2j, \\
           \end{array} \right.
X^{\alpha^-_{k,j}}\cdot z_r= \left\{ \begin{array}{ll}
             -2 z_{2j-1}, & r=2k-1, \\
             -2 z_{2j}, & r=2k, \\
             2 z_{2k-1}, & r=2j-1, \\
             2 z_{2k}, & r=2j, \\ \\
           \end{array} \right. \\ 
Y^{\alpha^+_{k,j}}\cdot z_r= \left\{ \begin{array}{ll}
             -2\i z_{2j}, & r=2k-1, \\
             2\i z_{2j-1}, & r=2k, \\
             2\i z_{2k}, & r=2j-1, \\
             -2\i z_{2k-1}, & r=2j, \\
           \end{array} \right.
Y^{\alpha^-_{k,j}}\cdot z_r= \left\{ \begin{array}{ll}
             -2\i z_{2j-1}, & r=2k-1, \\
             2\i z_{2j}, & r=2k, \\
             -2\i z_{2k-1}, & r=2j-1, \\
             2\i z_{2k}, & r=2j, \\ \\
           \end{array} \right. \\ 
X^{\alpha_i}\cdot z_r= \left\{ \begin{array}{ll}
             -2 z_m,& r=2i-1,2i, \\
              z_{2i-1}+z_{2i},& r=2l+1,
           \end{array} \right.
Y^{\alpha_i}\cdot z_r= \left\{ \begin{array}{ll}
             -2\i z_m, & r=2i-1, \\
              2\i z_m, & r=2i, \\
             -\i z_{2i-1}+\i z_{2i}, & r =2l+1.
           \end{array} \right. 
\end{array}
$$}
Let us fix a basis of $W_n$,
$$\mathcal{B}_1=\{ s p_{j_1,\dots, j_m}=s z_1^{j_1}\dots z_m^{j_m}, \quad s=1,\i, \;j_i \in \NN_0, \, j_i+\dots +j_m=n\},$$
\noindent where $z_i$ are defined in (\ref{zetas}) (see (\ref{basewm})) and let $V_1$ the subspace generated by $\mathcal{S}$,
$$\mathcal{S}=\{s z_j^n, \; s=1,\i,\; j \le 2l\},$$
\noindent and $V_2$ the subspace generated by $\mathcal{B}_1 \smallsetminus \mathcal{S}$. It is easy to see that the same arguments as in $\sug(m)$ case apply here to assert that the inner product that makes $\mathcal{B}_1$ an orthonormal basis of $W_n$ satisfies the hypotheses of Theorem \ref{theTheo}.
\

To tackle the  $\sog(2l)$ case, first note that the elements of  $\sog(2l)$ can be realized in much the same way as in $\sog(2l+1)$, we just have to erase the $l+1$ column and row. In particular, the root structure of $\sog(2l)$ can be read off from the one we have constructed for $\sog(2l+1)$ (see (\cite[Example 4 pp. 64]{Kn2}). In fact, one can choose the Cartan subalgebra so that the roots and root vectors of $\sog(2l)$ correspond to the ones that can be restricted from $\sog(2l+1)$. Explicitly, the set of roots is $\Delta = \{\pm e_k \pm e_j, \; j\le l\}$ and the corresponding root vectors are obtained from the ones in $\sog(2l+1)$ by erasing the last column and row. Hence, everything works in the same way.

\begin{theorem}
  Let $(W_n, \pi_n)$ be the standard real representation of $\sog(m)$ on the space of complex-valued homogeneous polynomials of degree $n$ on $\RR^m$. Let $(\lgo, \lb) = (\RR Z \oplus \sog(m)) \ltimes W_n$ be the Lie algebra where $[Z, \sog(m)]=0$ and $Z$ acts as the identity on $W_n$.
  Then $\lgo$ admits an inner product with negative Ricci curvature for all $n, m \ge 2$.
\end{theorem}

\begin{remark}
   We note that the root vectors used here are not the ones satisfying (\ref{sperootvect}) since  we have changed certain constants in order to simplify some calculations and expressions. It is easy to see that this is an equivalent realization.
\end{remark}

\subsection{Polynomial representations of $\spg(m)$}

Note that  $\spg(m)$ is a compact real form of the $C_m$-type complex Lie algebra $\spg(m, \CC)$.
Following \cite[Example 3, pp. 64]{Kn2}, we can choose as Cartan subalgebra
$$ \hg=\{ H \in \slg(2m, \CC): \, H= \Diag (h_1,\dots, h_m, -h_1,\dots,-h_m) \}.$$
For each $1 \le i \le m$, let $e_i \in \hg^*$ be defined by
$e_i(H)=h_i$ for any $H \in \hg$ as above. Then the system of roots is given by
$$\Delta =\{\pm e_i \pm e_j,\, \, i\ne j\} \cup \{\pm 2 e_k \}.$$
We choose an order so that the positive roots are
$$\Delta^+ =\{\alpha_{i,j}^\pm=e_i \pm e_j,\, \, i< j\} \cup \{\alpha_k=2 e_k \}, \text{ and let } \Pi=\{\alpha_{i,i+1}^-, \alpha_m,\,\, 1\le i\le m-1\},$$
\noindent be the simple ones.
Following \cite[Example 3, pp. 64]{Kn2}) we obtain a basis of $\spg(m)$ as in (\ref{genu}) given explicitly by
$$
\{H^{\alpha_{i,i+1}^-}, H^\alpha_{m},  X^{\alpha^\pm_{k,j}}, X^{\alpha_r}, Y^{\alpha^\pm_{k,j}}, Y^{\alpha_r}, \; i \le m-1,\,r \le m, k<j \le m\},
$$
\noindent where
\small{ $$ \begin{array}{l}
H^{\alpha_{i,i+1}^-}= \i E_{i,i}- \i E_{i+1,i+1} - \i E_{m+i,m+i}+\i E_{m+i+1,m+i+1} ,\, \quad H^{\alpha_m}= \i E_{m,m}-\i E_{2m,2m},\\ \\
X^{\alpha^-_{k,j}} = E_{k,j}-E_{m+j,m+k} - E_{j,k} + E_{m+k,m+j}, \, X^{\alpha^+_{k,j}} = E_{k,m+j}+E_{j,m+k} - E_{m+k,j} - E_{m+j,k}, \\ \\
Y^{\alpha^-_{k,j}} = \i E_{k,j} - \i E_{m+j,m+k} + \i E_{j,k} - \i E_{m+k,m+j}, \\ \\ Y^{\alpha^+_{k,j}} = \i E_{k,m+j}+ \i E_{j,m+k} + \i E_{m+k,j} + \i E_{m+j,k},    \\ \\
 X^{\alpha_r} = E_{r,m+r}-E_{m+r,r}, \quad  Y^{\alpha_r} = \i E_{r,m+r} + \i E_{m+r,r}  .
 \end{array}$$}

Let $(W_n, \pi_n)$ be the standard representation of $\spg(m)$ on the space of complex homogeneous polynomials of degree $n$ derived from the standard action of the group $\Sp(m,\CC)$. In this case we have that $(W_n, \pi_n)$ or $\Sym^n(\CC^{2m})$ is the irreducible representation of $\spg(m,\CC)$ of maximal weight $n\,e_1$ (see \cite{FH} \S 17.2). As in the previous cases, we will still denote by $(W_n, \pi_n)$ the corresponding real representation of $\ug = \spg(m)$.
Let us fix a basis of $W_n$,
\begin{equation}
\mathcal{B}_1= \{p_{j_1,\dots,j_{2m}},\; \i p_{j_1,\dots,j_{2m}}, \quad j_i \in \NN_0,\;\; j_1+\dots+j_{2m} =n \},
\end{equation}
where $p_{j_1,\dots,j_{2m}} = z_1^{j_1}\dots z_{2m}^{j_{2m}} \in \mathcal{P}_{n}(\CC^{2m})$.

Straightforward calculations show that the non-trivial action of $\spg(m)$ on $W_n$  can be obtained from
{\small $$ \begin{array}{l}
H^{\alpha_{i,i+1}^-}\cdot z_r= \left\{ \begin{array}{ll}
             -\i z_r, &  r=i,m+i+1,  \\
             \i z_r & r=i+1,m+i,
                  \end{array}  \right., \quad
H^{\alpha_m}\cdot z_r= \left\{ \begin{array}{ll}
             -\i z_r, &  r=m,  \\
             \i z_r & r=2m,
               \end{array}  \right. \\ \\
X^{\alpha^+_{k,j}}\cdot z_r= \left\{ \begin{array}{ll}
              z_{j}, & r=m+k, \\
             - z_{m+j}, & r=k, \\
               z_{k}, & r=m+j, \\
              -z_{m+k}, & r=j, \\
           \end{array} \right.
X^{\alpha^-_{k,j}}\cdot z_r= \left\{ \begin{array}{ll}
              z_{j}, & r=k, \\
              z_{m+j}, & r=m+k, \\
             -z_{k}, & r=j, \\
              - z_{m+k}, & r=m+j, \\ \\
           \end{array} \right. \\ 
Y^{\alpha^+_{k,j}}\cdot z_r= \left\{ \begin{array}{ll}
             -\i z_{j}, & r=k, \\
             -\i z_{m+j}, & r=m+k, \\
             -\i z_{k}, & r=j, \\
             -\i z_{m+k}, & r=m+j, \\
           \end{array} \right.
Y^{\alpha^-_{k,j}}\cdot z_r= \left\{ \begin{array}{ll}
             -\i z_{j}, & r=k, \\
             \i z_{m+j}, & r=m+k, \\
             -\i z_{k}, & r=j, \\
             \i z_{m+k}, & r=m+j, \\ \\
           \end{array} \right. \\
X^{\alpha_i}\cdot z_r= \left\{ \begin{array}{ll}
             -z_{m+i},& r=i, \\
              z_{i},& r=m+i,
           \end{array} \right.
Y^{\alpha_i}\cdot z_r= \left\{ \begin{array}{ll}
             -\i z_{m+i}, & r=i, \\
              -\i z_i, & r=m+i.
           \end{array} \right.
\end{array}
$$}
Let $V_1$ be the subspace generated by $\mathcal{S}$,
$$\mathcal{S}=\{s z_j^n, \; s=1,\i,\; j \le 2m\},$$
\noindent and $V_2$ the subspace generated by $\mathcal{B}_1 \smallsetminus \mathcal{S}$. In the same way as in the case of $\sug(m)$ one can see that the the inner product that makes $\mathcal{B}_1$ an orthonormal basis of $W_n$ satisfies the hypotheses of Theorem \ref{theTheo}.

\

\begin{theorem}
  Let $(W_n, \pi_n)$ be the standard real representation of $\spg(m)$ on the space of complex homogeneous polynomials of degree $n$. Let $(\lgo, \lb) = (\RR Z \oplus \spg(m)) \ltimes W_n$ be the Lie algebra where $[Z, \spg(m)]=0$ and $Z$ acts as the identity on $W_n$.
  Then $\lgo$ admits an inner product with negative Ricci curvature for all $n, m \ge 2$.
\end{theorem}

\subsection{More examples starting from a non compact $\ggo_0$}

   When $\ggo_0$ is a semisimple Lie algebra of complex matrices stable under $\theta$ where $\theta(X)=-\bar{X}^t$ and $\ggo_0 = \kg \oplus \pg$ is the corresponding Cartan decomposition such that $\kg \cap \i \pg =0$,  one obtains that its complexification $\ggo = (\kg \oplus \pg)^\CC$ is also semisimple and $\ug=\kg \oplus \i\pg$ is a compact real form of $\ggo$. As we have already mentioned, in this case $\ggo_0$ is usually called the non-compact dual of $\ug$.  In particular, any finite-dimensional complex representation of $\ggo_0$ gives rise to a representation of $\ug$ and viceversa by using this decomposition (see \cite{Kn} pag. 443).
Hence, we can follow the same construction using $\ggo_0$ instead of $\ug$, both realized as subalgebras of complex matrices. That is, to start from the basis given by the roots of $\ggo$, consider the semidirect product $\lgo= (\RR Z \oplus \ggo_0) \ltimes V$, apply the degeneration and get to a solvable Lie algebra $\lgo_\infty$ which, for some representations $V$ could have negative Ricci operator. The calculations are more involved since the operators $\ad H_\alpha$ are no longer skew-symmetric and the basis of $\ngo$ is no longer nice.

In the particular case when $\ggo_0=\slg(m,\RR)$ and $V=\mathcal{P}_{m}(\CC^{m})$, viwed as a real vector space, we can see that all this procedure works for the same choice of $V_1$.

\begin{proposition}\label{Gengl}
 Let $(W_n,\pi_n)$ be the usual real representation of $\slg(m,\RR)$ on the space of complex homogeneous polynomials of degree $n$ in $m$ variables, extended to $\glg(m,\RR)$ by letting the center acts as multiples of the identity. Hence the Lie algebra $\glg(m,\RR)\ltimes V_n$ admits an inner product with negative Ricci curvature for any $m$ and $n \ge 2$.
\end{proposition}

In the next section, we will prove a more general result that implies Proposition \ref{Gengl} for
$m \ge 3$ so we will omit the proof. It is worth to point out that no perturbation of the inner product is needed for $m \ge 3$, since the basis
$$
\mathcal{B}=\{Z, H_\alpha, X^\beta, \i Y^\beta: \, \, \alpha \in \Pi, \beta \in \Delta^+ \} \cup \mathcal{B}_1
$$
\noindent (see  (\ref{basewm})), is a basis of eigenvectors of $\Ricci_\nu$ with negative eigenvalues for $m \ge 3, n \ge 2$, though is not nice. Although, these  results are a particular case of Theorem \ref{ssnc} for $m \ge 3$, this approach has the advantage of giving the inner product explicitly. As an example we will go over the example of $\glg(2,\RR)$ acting on $W_2=\mathcal{P}_{2}(\CC^2),$ the space of homogeneous complex polynomials of degree $2$ in $2$ variables seen as a real vector space. Note that for this particular example, the Levi factor of $\lgo= \glg(2,\RR) \ltimes W_2$ is $\slg(2,\RR)$, the only non-compact semisimple Lie algebra which is known to admit no $\Ricci < 0$ inner product (See \cite{M}).

\begin{example}\label{exegl2v1}
 Let $\ggo_0=\glg(2,\RR)$ and let $(W_2,\pi_2)$ be the representation of $\ggo$ on the space of homogeneous complex polynomials of degree $2$ in $2$ variables seen as a real vector space.
As in (\ref{formCom}), let
\begin{equation}\label{sl2}
  H=H_{\alpha_1}=\left[ \begin{smallmatrix} 1 & \\ & -1 \end{smallmatrix}\right], \quad X=X^{\alpha_{1,2}}=\left[ \begin{smallmatrix}  & 1 \\-1 &  \end{smallmatrix}\right], \quad
Y=-\i Y^{\alpha_{1,2}}=\left[ \begin{smallmatrix}  & 1\\ 1 &  \end{smallmatrix}\right],
\end{equation}
\noindent and $Z= \Id$.
We fix the orthonormal basis of $\lgo= \glg(2,\RR) \ltimes W_2 = (\RR^{10},\nu,\ip)$
$$\begin{array}{l}
\qquad \qquad \beta=\{Z, H, X, Y, v_1, v_2, v_3, v_4, v_5, v_6\},\\ \\
\text{ where } \qquad v_1=z_1^2,\; v_2= \i z_1^2,\; v_3=z_1z_2,\; v_4=\i z_1z_2,\; v_5=z_2^2,\;v_6=\i z_2^2.
\end{array}$$
We have
{\small $$
\begin{array}{l}
  \pi_2(Z) = \left[\begin{smallmatrix} 1 && \\ &\ddots&\\&&1\end{smallmatrix}\right], \qquad
\pi_2(H) = \left[\begin{smallmatrix}-2&&&&&\\&-2&&&&\\&&0&&&\\&&&0&&\\&&&&2&\\&&&&&2\end{smallmatrix}
   \right], \\ \\
  \pi_2(X) = \left[\begin{array}{cc:cc:cc}  &  & 1& & & \\ &&&1&& \\ \hdashline  -2& &&& 2& \\ &-2&&&&2 \\ \hdashline  & & -1&&& \\ &&&-1&&  \end{array} \right], \quad
\pi_2(Y) = \left[\begin{array}{cc:cc:cc} &  & -1& & & \\ &&&-1&& \\ \hdashline  -2& &&& -2& \\ &-2&&&&-2 \\ \hdashline  & & -1&&& \\ &&&-1&&  \end{array} \right].
\end{array}
$$}
In this case $V_1$ is the subspace generated by $\mathcal{S}=\{v_1,v_2,v_5,v_6\}$ (see (\ref{DefS})) so we get
that the degeneration is given by $\phi_t \in \Gl(\sg)$
$${\phi_t}|_\ggo= \left[\begin{smallmatrix}1&&&\\&1&&\\&&t&\\&&&t  \end{smallmatrix}\right], \quad
{\phi_t}|_{V_2}= \left[\begin{smallmatrix}t&&&&&\\&t&&&&\\&&t^2&&&\\&&&t^2&&\\&&&&t&\\&&&&&t  \end{smallmatrix}\right],
$$
and hence, the limit $\lgo_\infty = (\RR^{10},\mu,\ip)$
is a solvable Lie algebra. Its nilradical is $\ngo = \Span \{X,Y\} \oplus V_2$ and the center of $\ngo$ is $\zg = \Span \{v_3,v_4 \}$.

Direct calculation shows that if we change the inner product so that
$$\{Z,H,X,Y,\unm v_1,\unm v_2, v_3,  v_4, \unm v_5, \unm v_6 \}$$
is an orthonormal basis, the corresponding Ricci operator is given by
\begin{equation}
\Ricci_\mu = \Diag(-6,-24,-2,-2,-7,-7,-4,-4,-7,-7).
\end{equation}

\end{example}


\begin{remark}{\rm
 As in the $\sug(2)$ case we can show that $\sg = \glg(2,\RR) \ltimes \CC^2,$ that is when
 $\slg(2,\RR)$ acts on $W_1=\CC^2$ seen as a real vector space, also admits an inner
 product with negative Ricci curvatures. This is the analogous of \cite[ Lemma 3.4]{u2}  so as in
 that case we only need to consider a slightly different degeneration and the right
 change of basis and therefore we will just give very few details.

Consider the metric Lie algebra $\lgo_\infty = (\RR^8, \mu, \ip),$ where $\ip$ is the inner product that makes
$$\mathcal{B}=\{Z, H, X, Y, z_1, \i z_1, z_2, \i z_2\}$$ an orthonormal basis and the family $\phi_t$ as in \cite[ Lemma 3.4]{u2}.
Direct calculation shows that
$$ \Ricci_{\mu} = \left[ \begin{smallmatrix}
 -4&&&&&&&\\&-12&&&&&&\\&&-1&1&&&&\\&&1&-1&&&&\\&&&&-5&&&\\&&&&&-5&&\\&&&&&&-3&\\&&&&&&&&-3
\end{smallmatrix}\right].$$
By changing the inner product so that
$$\{Z, H, X+Y, X-Y, z_1, \i z_1, z_2, \i z_2\}$$
\noindent is an orthonormal basis, we get
$$ \Ricci_{\mu} = \Diag (-4,-12,-8,-12,-2,-2,-6,-6),$$
\noindent as desired. Then $\lgo_\infty$ and therefore $\lgo$ both admit an inner product with negative Ricci curvature.
}\end{remark}

\section{A more general Construction.}

In this section, we obtain a generalization of the construction in the previous section in the sense that we consider  more general semidirect products to find examples of non-solvable Lie groups with negative Ricci curvature.
We construct Lie algebras $\ggo= \ag \oplus \rg \oplus \ngo$, where $\rg$ is a semisimple Lie algebra without compact factors, $\ngo$ is a nilpotent ideal and $\ag$ is abelian.
In \cite{LL}, the Ricci operator for homogeneous spaces has been studied. We will use some of their ideas and notation since many of the formulas used there are general.

\begin{definition} In the following, we will denote by $\ggo = \ggo(\ag,\rg,\ngo) = (\ag \oplus \rg) \ltimes \ngo$ a Lie algebra such that
\begin{itemize}
 \item $\rg$  is semisimple with no compact factors,
 \item $\ag$  is abelian,
 \item $\ngo$  is nilpotent,
 \item $[\ag, \rg]=0.$
\end{itemize}
\end{definition}

Fix $\ip$ any inner product on $\ggo$ that makes $\ag \oplus\rg \oplus \ngo$ an orthogonal decomposition.
We note that the mean curvature vector $H$ is orthogonal to $\ngo$ and to $\hg$ so $H \in \ag$.
Since $\ag$ is abelian, $\rg$ is a subalgebra and $\ngo$ is a nilpotent ideal, using formulas from \cite[Lemma 4.4]{LL}, we can show that the Ricci operator of  $(\ggo, \ip)$ is given by

\begin{equation}\label{Ricgennc}
\begin{array}{l}
\langle \Ricci Y, Y\rangle = \langle\Ricci_\hg Y, Y \rangle - \tr S(\ad Y|_\ngo)^2,\\ \\
\langle \Ricci A, A\rangle = - \tr S(\ad A|_\ngo)^2, \\ \\
\Ricci|_\ngo = \Ricci_\ngo - S(\ad H|_\ngo) + \unm \sum [\ad Y_i|_\ngo, (\ad Y_i|_\ngo)^t] + \unm \sum [\ad A_i|_\ngo, (\ad A_i|_\ngo)^t],\\ \\
\langle \Ricci Y, A\rangle = - \tr S(\ad Y|_\ngo) S(\ad A|_\ngo),\\ \\
\langle \Ricci Y, X\rangle = - \tr (\ad Y|_\ngo)^t (\ad_\ngo X),\\ \\
\langle \Ricci A, X\rangle = -  \tr (\ad A|_\ngo)^t (\ad_\ngo X),
  \end{array}
  \end{equation}
where $Y \in \rg$, $A \in \ag$, $X \in \ngo$, $\{Y_i\}$ $\{A_i\}$ are orthonormal basis of $\rg$ and $\ag$ respectively and by $\Ricci_\lgo$ we denote the Ricci operator of the Lie subalgebra $\lgo$ with the restricted inner product.

\begin{theorem}\label{ssnc}
Let $\ggo= \ggo(\ag,\rg,\ngo)$ be as above. If in addition,
\begin{enumerate}
    \item $\ngo$ admits an inner product such that $\ad A|_\ngo$ are normal operators for any $A \in \ag$,
  \item no $\ad A|_\ngo$ has all its eigenvalues purely imaginary,
  \item there exists an element $A_1$ in $\ag$ such that all the eigenvalues of $\ad A_1|_\ngo$ have positive real parts,
  \item\label{4}  $\rg$ admits an inner product with orthogonal Cartan decomposition $\rg= \kg \oplus \pg$ and with $\Ricci <0$,
\end{enumerate}
then $\ggo$ admits an inner product with negative Ricci curvature.
\end{theorem}

 Some observations before the proof. First, by results of \cite{JP}, it is known that a semisimple Lie group that admits a negative Ricci curved metric can not have any compact factors (see (\ref{4})).

It is easy to see that since $\ag$ is abelian it is enough to check the first hypothesis in a basis.

Also note that if $\rg$ is a non-compact semisimple Lie algebra with Cartan decomposition $\rg=  \kg \oplus \pg$ and $\ip$ is an inner product on $\rg$ such that $\kg$ is orthogonal to $\pg$ (as in (\ref{4})) then $\Ricci(\kg,\pg) = 0$. Indeed, $\Hm=0$ as $\rg$ is unimodular and there exists an inner product on $\rg$ such that $\ad X$ is a symmetric operator  for $X \in \pg$ and $\ad Y$ is skew-symmetric for $Y \in \kg$, then the Killing form satisfies $\langle BX, Y\rangle=0$ for any $X \in \pg$ and $Y\in \kg$. Also, from (\ref{R}) we get that for any orthonormal basis of $\rg$ $\{X_i\}$,
\begin{equation}\label{kortp}
\la MX,Y\ra=-\unm\sum\la [X,X_i],X_j\ra \la [Y,X_i],X_j\ra +\unc\sum\la
[X_i,X_j],X\ra \la
[X_i,X_j],Y\ra.
\end{equation}
Let us chose an $\{X_i\}$, so that the first elements are in $\kg$ and the last ones are in $\pg$. Hence, if $X \in \pg$ and $Y\in \kg$, using that
$$[\kg,\kg] \subset \kg,\; [\kg,\pg] \subset \pg, \; [\pg, \pg] \subset \kg,$$
it is easy to check that all the terms in (\ref{kortp}) vanish and therefore $\Ricci(\kg,\pg) = 0$.

 Also note that if  $\ggo$ is a Lie algebra we can consider its Levi decomposition $\ggo= \rg \oplus \sg$ and decompose the radical $\sg$ as $\sg=\ag \oplus \ngo$ where $\ngo$ is the nilradical of $\ggo$. It is not hard to see that there always exists a complement $\ag$ which satisfies $[\ag, \rg]=0$, therefore that hypothesis is not so restrictive.  In fact, since $\sg$ is an ideal and $\rg$ is non-compact semisimple Lie algebra, there exists an inner product $\ip$ on $\sg$ and a basis of $\rg$ $\beta$, such that $\ad Y: \sg \to \sg$ are symmetric or skew-symmetric operators for any $Y \in \beta$ (see (\ref{innerpro})).  Also, for any $Y \in \beta$, $\ad Y$ is a derivation of $\sg$ and hence $\ad Y(\sg) \subset \ngo$ (see \cite[Lemma 2.6]{GorWi}). Let $\ag$ be the orthogonal complement of $\ngo$ and hence, since $\ngo$ is an ideal, $\ad Y(\ngo) \subset \ngo$ and therefore for any $A \in \ag, X \in \ngo$
$$\langle [Y, A], X \rangle = \pm \langle A, [Y, X] \rangle = 0,$$
for any $Y \in \beta$ and therefore for any $Y \in \rg$.

Finally, note that to get (1) it is enough to have that $\ad A|_\ngo$ are semisimple operators for any $A \in \ag$. Indeed, if $\ad A|_\ngo$ are semisimple operators for any $A \in \ag$ since they commute, there exists a basis of $\ngo^\CC$, $\beta$, of common eigenvectors of $\{\ad A|_\ngo,\, A \in \ag\}$. Then it is easy to see that they are normal operators of $(\ngo, \ip_0)$ where $\ip_0$ is the real part of the inner product that makes $\beta$ an orthonormal basis of $\ngo^\CC$.

\begin{proof}

We will consider first the case when $\ngo$ is abelian.
Let $\ip$ be an inner product on $\ggo$ such that the decomposition $\ag \oplus \rg \oplus \ngo$ is orthogonal,  $\ip|_\rg$ is an inner product with negative Ricci curvature, $\rg=  \kg \oplus \pg$ is the orthogonal Cartan decomposition and $\ip|_\ngo$ is the inner product given in (1). Note that since $\ngo$ is abelian,  $\Ricci(\ag,\ngo)=0$ and  $\Ricci(\rg,\ngo)=0$ (see (\ref{Ricgennc})).

Since $[\rg, \ag]=0$ we can also assume with no loss of generality, that the elements in $\kg$ act on $\ngo$ by skew-symmetric operators and $\pg$ acts by symmetric ones. Indeed, let $R$ be the complex simply connected Lie group with Lie algebra $\rg^\CC$ and let $R_1$ be the connected Lie subgroup of $R$ with Lie algebra $\rg_1 = (\kg + \i \pg)$. Since $\rg_1$ is compact, for any $\ip_0$  hermitian form on $\ngo^\CC$  we get that
 \begin{equation}\label{innerpro}
   \langle X, X' \rangle_1 = \int_{R_1} \langle \pi(h) (X), \pi(h)  (X')\rangle_0\; dh,
 \end{equation}

 \noindent defines an $R_1$-invariant hermitian form on $\ngo^\CC$, where $\pi$ is the representation of $R$ on $\ngo^\CC$ such that $d \pi = \ad|_\ngo$. Using that $R_1$ is connected and the fact that $[\rg_1, \ag]=0$, we get that $\ad A$ commutes with $\ad(Y)=d \pi(Y)$ for any $Y \in \rg_1$ and therefore with $\pi(h)$ for any $h \in R_1$. Hence, if $\ad A$ is normal with respect to $\ip_0$, its transpose (with respect to $\ip_0$), $(\ad A)^t_0$ also commutes with $\pi(h)$. Using this in (\ref{innerpro}) it is easy to see that the transpose of $\ad A$ is the same for both hermitian forms $\ip_1$, $\ip_0$ and consequently if $\ad A$ is normal with respect to $\ip_0$, it is also a normal operators of $(\ngo^\CC, \ip_1)$. Finally, consider the real part of $\ip_1$, $\ip$ which is an inner product on $\ngo$ with the desired properties.


We therefore get form (\ref{Ricgennc})
\begin{equation}\label{Richomneg}
\begin{array}{l}
\langle \Ricci Y, Y\rangle = \langle \Ricci_\hg Y, Y\rangle - \tr S(\ad Y|_\ngo)^2, \\ \\
\langle \Ricci A, A\rangle =- \tr S(\ad A|_\ngo)^2,  \\ \\
\Ricci|_\ngo = - S(\ad H|_\ngo) + \unm \sum [\ad Y_i|_\ngo, (\ad Y_i|_\ngo)^t], \\ \\
\langle \Ricci Y, A\rangle = - \tr S(\ad Y|_\ngo) S(\ad A|_\ngo),\\ \\
\langle \Ricci Y, X\rangle = 0,  \qquad  \langle \Ricci A, X\rangle = 0,
  \end{array}
  \end{equation}
 for $Y \in \rg$, $A \in \ag$, $X \in \ngo$.
Chose an orthonormal basis of $\rg$, $\{Y_1, \dots,Y_r\}$ so that the $Y_j \in \kg$ for $j \le m$ and $Y_j \in \pg$ for $j > m$. Therefore, $[\ad Y_i|_\ngo, (\ad Y_i|_\ngo)^t] = 0$.

Let $\{A_1, \dots, A_k\}$ be a basis of $\ag$ so that $A_1$ is the element as in the statement
and $\tr \ad A_i=0$ for all $i \ge 2$ and take the inner product that makes this an orthonormal basis of $\ag$. Note that up to now the inner product on $\ag$ has no conditions. From this, it is easy to see that $H = \tr ({\ad A_1}|_\ngo) A_1.$

We note that by the choice of the basis, $\langle \Ricci Y_j, A_k\rangle = 0$ for any $j \le m$.
Also, since $\{(\ad A_i|_\ngo), 1 \le i \le r\}$ is a set of normal operators that commutes with each other, $(\ad A_i|_\ngo)$ commutes with $(\ad A_j|_\ngo)^t$ for each $i,j$. Thus,  $\{S(\ad A_i|_\ngo)\}$ is a set of symmetric operators that commutes with each other and therefore there exists $\beta=\{X_1, \dots X_n\}$ a basis of $\ngo$ of common eigenvectors of $\{S(\ad A_i|_\ngo), 1 \le i \le r\}$. Hence, consider a decomposition $\ngo = \ngo_1 \oplus \dots \oplus \ngo_p$ such that $S(\ad A_i|_{\ngo_j}) = a_{ij} \, Id_{\ngo_j}$, $a_{ij} \in \RR$.
Now, since $[Y, A]=0$, $\ad Y|_\ngo$ preserves the subspaces $\ngo_j$ and therefore we obtain that, for any $j > m,$
$$ \begin{array}{ll}
 \langle \Ricci Y_j, A_k\rangle & = \tr S(\ad {Y_j}|_\ngo)S(\ad {A_k}|_\ngo) = \sum  \tr(\ad {Y_j}|_{\ngo_l})S(\ad {A_k}|_{\ngo_l})\\ \\
  & = \sum  \tr(\ad {Y_j}|_{\ngo_l})(a_{kl} \,\Id_{\ngo_l}) = 0,
\end{array}
$$
since $\hg$ is semisimple.
Thus, from (\ref{Richomneg})
$$\begin{array}{l}
\langle \Ricci Y, Y\rangle = \langle \Ricci_\hg Y, Y\rangle - \tr S(\ad Y|_\ngo)^2,  \\ \\
\langle \Ricci A, A\rangle =- \tr S(\ad A|_\ngo)^2,  \\ \\
\Ricci|_\ngo = - \tr({\ad A_1}|_\ngo) S(\ad A_1|_\ngo), \\ \\
\langle \Ricci Y, A\rangle = 0, \quad \langle \Ricci Y, X\rangle = 0,  \quad  \langle \Ricci A, X\rangle = 0,
\end{array}
$$
is negative definite, as was to be shown.

If $\ngo$ is not abelian, for each $t>0$ consider $\psi_t \in \glg(\ggo)$ such that
  $$
  {\psi_t}|_{\ag \oplus \rg}=\Id, \qquad {\psi_t}|_{\ngo}=t \,\Id.
  $$
  It is easy to check that $\lb_t = \psi_t \cdot \lb$ is given by
\begin{equation}\label{degabel}
\begin{array}{l}
    [X_1,X_2]_t=[X_1,X_2] \quad  \text{ for } X_i \in \ag\oplus \rg,\; i=1,2, \\ \\

    [X_1,X_2]_t=\tfrac{1}{t}[X_1,X_2] \quad \text{ for } X_i \in \ngo,\; i=1,2,\\ \\

     [X_1,X_2]_t=[X_1,X_2] \quad \text{ for } X_1 \in \ag\oplus \rg,\; X_2 \in \ngo.
  \end{array}
\end{equation}

  In the last two equations we have used that $\ngo$ is an ideal.
  Hence, $\displaystyle{\lim_{t \to \infty}} \lb_t =\mu_0$ is well-defined and is given by
    $$\begin{array}{l}
    \mu_0(X_1,X_2)=[X_1,X_2],  \, X_1 \in \ag \oplus \rg, X_2 \in \ggo,\quad   \mu_0(X_1,X_2)=0,  \, X_i  \in \ngo.
    \end{array} $$
Therefore, the limit Lie algebra satisfies the same conditions as in the statement and $\ngo$ is now abelian. Using the previous results, the limit Lie algebra admits an inner product with negative Ricci curvature and therefore, by Proposition \ref{limit} so does $\ggo$.
\end{proof}

\begin{remark}
 We note that since $\rg$ is semisimple with no compact factors, these examples are completely different from examples arising from Theorem $\ref{theTheo}$. On the other hand,  when $\ag=\RR Z$ is acting as a multiple of the identity and we consider one of the inner products on $\slg(n,\RR)$ for $n \ge 3$ given in \cite{DL} we get the results of the previous section for any real representation. Recall that in \cite{DLM} it is shown that most of the simple non-compact Lie algebras admits an inner product satisfying the properties in Theorem \ref{ssnc} and from there we get a lot of examples.
\end{remark}

Finally, coming back to the compact case, we can use the same idea as in the previous theorem  to get examples with a non-abelian $\ngo$. Since we only have a complete description in the case when $\ggo=\sug(2)$, we will only state the result in that case. Note that to get a Lie algebra, $\sug(2)$ should act by derivations on $\ngo$.

\begin{theorem}
  Let $\ggo= (Z\RR \oplus \sug(2)) \ltimes \ngo$ be a Lie algebra where $\ngo$ is any nilpotent Lie algebra and $[Z,\sug(2)]=0$. Let $\pi=ad|_{\sug(2)}$ acting on $\ngo$ and let $\ngo=\ngo_1\oplus \dots \oplus \ngo_k$ be the decomposition of $\ngo$ in irreducible components for $\pi$. If $\pi$ is not trivial and $Z$ acts in each $\ngo_i$ as a positive multiple of the identity, then $\ggo$ admits an inner product with negative Ricci curvature.
\end{theorem}

\begin{proof}
  Let $(\ggo, \lb)$ be the Lie algebra as defined above and endow it with the inner product such that $\lVert Z \rVert=1$ and $Z\RR \oplus \sug(2) \oplus \ngo$  is an orthogonal decomposition. Let $\psi_t \in \glg(\ggo)$ be as in (\ref{degabel}) where $\ag = \RR Z$ and $\rg=\sug(2)$. Hence, as it was shown in the previous theorem,  $\mu_o=\displaystyle{\lim_{t \to \infty}} \psi_t \cdot \lb$ is well defined and it is given by
    $$\begin{array}{l}
    \mu_0(X_1,X_2)=[X_1,X_2],  \, X_1 \in  \RR Z \oplus \sug(2), X_2 \in \ggo,\quad   \mu_0(X_1,X_2)=0,  \, X_i  \in \ngo.
    \end{array} $$
   Let $\pi = \ad|_{\sug(2)}$ acting on $\ngo$ and decompose the linear space $\ngo$ in irreducible components for the action of $\pi$. Note that $Z$ acts as a positive multiple of the identity in each $\ngo_i$.  Now we can follow the same proof given in \cite{u2} with a few little differences since the mean curvature vector is different (see \cite[Remark 3.13]{u2}).
  \end{proof}

\end{document}